\documentclass{article}%
\usepackage{amsmath,amsfonts,amsthm,amssymb, graphicx}
\usepackage{subfig}
\usepackage{amsmath}
\usepackage{amsfonts}
\usepackage{amssymb}
\usepackage{graphicx}
\usepackage{sidecap}
\usepackage{array}
\newcolumntype{P}[1]{>{\centering\arraybackslash}p{#1}}
\usepackage{color}
\usepackage{mathrsfs}%
\newtheorem{theorem}{Theorem}

\newtheorem{lemma}{Lemma}

\newtheorem{proposition}{Proposition}

\newtheorem{assumption}{Assumption}
\begin{document}

\title{Perfect Sampling of $GI/GI/c$ Queues}
\author{Blanchet, J., Dong, J., and Pei, Y.}
\date{}
\maketitle

\begin{abstract}
We introduce the first class of perfect sampling algorithms for the
steady-state distribution of multi-server queues with general interarrival
time and service time distributions. Our algorithm is built on the classical
dominated coupling from the past protocol. In particular, we use a coupled
multi-server vacation system as the upper bound process and develop an
algorithm to simulate the vacation system backwards in time from stationarity
at time zero. The algorithm has finite expected termination time with mild
moment assumptions on the interarrival time and service time distributions.

\end{abstract}

\section{Introduction}

In this paper, we present the first class of perfect sampling algorithms for
the steady-state distribution of multi-server queues with general interarrival
time and service time distributions. Our algorithm has finite expected running
time under the assumption that the interarrival times and service times have
finite $2+\epsilon$ moment for some $\epsilon>0$.

The goal of perfect sampling is to sample without any bias from the
steady-state distribution of a given ergodic process. The most popular perfect
sampling protocol, known as Coupling From The Past (CFTP), was introduced by
Propp and Wilson in the seminal paper \cite{PropWil:1996}; see also
\cite{AsmGlyTho:1992} for another important early reference on perfect simulation.

Foss and Tweedie \cite{FosTwe:1998} proved that CFTP can be applied if and
only if the underlying process is uniformly ergodic, which is not a property
applicable to multi-server queues. So, we use a variation of the CFTP protocol
called Dominated CFTP (DCFTP) introduced by Kendall in \cite{Kendall:1998} and
later extended in \cite{KendallMoller:2000, Ken:2004}.

A typical implementation of DCFTP requires at least four ingredients:

\begin{itemize}
\item[a)] a stationary upper bound process for the target process,

\item[b)] a stationary lower bound process for the target process,

\item[c)] the ability to simulate a)\ and b) backwards in time (i.e. from time
$[-t,0]$ for any $t>0$)

\item[d)] a finite time $-T<0$ at which the state of the target process is
determined (typically by having the upper and lower bounds coalesce), and the
ability to reconstruct the target process from $-T$ up to time $0$.
\end{itemize}

The time $-T$ is called the coalescence time and it is desirable to have
$E\left(  T\right)  <\infty$. The ingredients are typically combined as
follows. One simulates a) and b) backwards in time (by applying c)) until the
processes meet. The target process is sandwiched between a) and b). Therefore,
if we can find a time $-T<0$ when processes a) and b) coincide, the state of
the target process is known at $-T$ as well. Then, applying d), we reconstruct the
target process from $-T$ up to time $0$. The algorithm outputs the state of the
target process at time $0$.

It is quite intuitive that the output of the above construction is
stationary.
Specifically, assume that the sample path of the target process, coupled with,
a)\ and b) is given from $(-\infty,0]$,
Then we can think of the simulation procedure in c) as
simply observing or unveiling the paths of a) and b) during $[-t,0]$. When we
find a time $-T<0$ at which the paths of a) and b) take the same value,
because of the sandwiching property, the target process must share this
common value at $-T$. Starting from that point, property d) simply
unveils the path of the target process. Since this path has
been coming from the infinite distant past (we simply observed it from time
$-T$), the output is stationary at time 0.

One can often improve the performance of a DCFTP protocol if the underlying
target process is monotone \cite{Ken:2004}, as in the multi-server queue
setting. A process is monotone if there exists a certain partial order, $\preceq$,
such that if $w$ and $w^{\prime}$ are initial states where $w\preceq
w^{\prime}$, and one uses common random numbers to simulate two paths, one
starting from $w$ and the other from $w^{\prime}$, then the order is
preserved when comparing the states of these two paths at any point in time.
Thus, instead of using the bounds a) and b) directly to detect coalescence,
one could apply monotonicity to detect coalescence as follows. At any time
$-t<0$, we can start two paths of the target process, one from the state
$w^{\prime}$ obtained from the upper bound a) observed at time $-t$, and the other from
the state $w\preceq w^{\prime}$ obtained from the lower bound b) observed at time $-t$.
Then we run these two paths using the common random numbers, which are consistent with
the backwards simulation of a) and b), in reverse order according to the
dynamics of the target process, and check if these two paths meet before time
zero. If they do, the coalescence occurs at such meeting time. We also notice
that because we are using common random numbers and system dynamics, these two
paths will merge into a single path from the coalescence time forward, and the
state at time zero will be the desired stationary draw. If coalescence does
not occur, then one can simply let $t\longleftarrow2t$, and repeat the above procedure. For
this iterative search procedure, we must show that the search terminates in
finite time.

While the DCFTP protocol is relatively easy to understand, its application is
not straightforward. In most applications, the most difficult part has to do
with element c). Then, there is the issue of finding good bounding processes
(elements a) and b)), in the sense of having short coalescence times - which
we interpret as making sure that $E\left(  T\right)  <\infty$. There has been
a substantial amount of research which develops generic algorithms for Markov
chains (see for example \cite{CorTwe:2001} and \cite{ConKen:2007}).
These methods rely on having access to the transition kernels which is
difficult to obtain in our case. Perfect simulation for queueing systems
has also received significant amount of attention in recent years, though most
perfect simulation algorithms for queues impose Poisson assumptions on the
arrival process. Sigman \cite{Sig:2011, Sig:2012} applied the DCFTP and
regenerative idea to develop perfect sampling algorithms for stable $M/G/c$
queues. The algorithm in \cite{Sig:2011} requires the system to be
super-stable (i.e. the system can be dominated by a stable $M/G/1$ with). The
algorithm in \cite{Sig:2012} works under natural stability conditions, but it
has infinite expected termination time. A recent work by Connor and Kendall
\cite{ConKen:2014} extends Sigman's algorithm \cite{Sig:2012} to sample
stationary $M/G/c$ queues and the algorithm has finite expected termination time, but they
still require the arrivals to be Poisson. The main reason for the Poisson
arrival assumption is that under this assumption one can find dominating
systems which are quasi-reversible (see Chapter 3 of
\cite{KellyReversibility:1979})
and therefore can be simulated backwards in time using standard Markov chain
constructions (element c)).

For general renewal arrival process, our work is close in the spirit to
\cite{EnsGlyn:2000}, \cite{BlanChen:2013}, \cite{BlanDong:2014} and
\cite{BlanWall:2014}, but the model treated is fundamentally different Thus,
it requires some new developments. We also use a different coupling
construction to that introduced in \cite{Sig:2012} and refined in
\cite{ConKen:2014}. In particular, we take advantage of a vacation system
which allows us to transform the problem into simulating the running infinite
horizon maximum (from time $t$ to infinity) of renewal processes, compensated
with a negative drift so that the infinite horizon maximum is well defined.
Finally, we note that a significant advantage of our method, in contrast to
\cite{Sig:2012} is that {we do not need to empty the system in order to
achieve coalescence. This is important in many server queues in heavy traffic
for which it would take an exponential amount of time (in the arrival rate) or
sometimes impossible to observe an empty system.}

The rest of the paper is organized as follows. In Section \ref{sec:main} we
describe our simulation strategy, involving elements a) to d), and we conclude
the section with the statement of a result which summarizes our main
contribution (Theorem \ref{Thm_Main_1}). Subsequent sections (Section
\ref{sec:vacation}, \ref{sec:coalescence} \& \ref{sec:algorithm}) provide more
detailed justification for our simulation strategy. Lastly we conduct
numerical experiments in Section \ref{sec:numerical}. An online companion of
this paper includes a Matlab implementation of the algorithm.


\section{Simulation Strategy and Main Result\label{sec:main}}

Our target process is the stationary process generated by a multi-server queue
with independent and identically distributed (iid) interarrival times and iid
service times which are independent of the arrivals. There are $c\geq1$
identical servers, each can serve at most one customer at a time. Customers
are served on a first-come-first-served (FCFS) basis. Let $G(\cdot)$ and
$\bar{G}(\cdot)=1-G(\cdot)$ (resp. $F(\cdot)$ and $\bar{F}(\cdot)=1-F(\cdot)$)
denote the cumulative distribution function, CDF, and the tail CDF of the
interarrival times (resp. service times). We shall use $A$ to denote a random
variable with CDF $G$, and $V$ to denote a random variable with CDF $F$.

\begin{assumption}
[A1]Both $A$ and $V$ are strictly positive with probability one and there
exists $\epsilon>0$, such that
\[
E[A^{2+\epsilon}]<\infty,\text{ \ \ }E[V^{2+\epsilon}]<\infty.
\]

\end{assumption}

The previous assumption will allow us to conclude that the coalescence time of
our algorithm has finite expectation. The algorithm will terminate with
probability one if $E[A^{1+\epsilon}]+E[V^{1+\epsilon}]<\infty$. The
assumption of $A$ and $V$ being strictly positive can be done without loss of
generality because one can always reparametrize the input in cases where
either $A$ or $V$ has an atom at zero.

We assume that $G\left(  \cdot\right)  $ and $F\left(  \cdot\right)  $ are
known so that the required parameters in Section 3.1.1 of \cite{BlanWall:2014}
can be obtained. We write $\lambda=(\int_{0}^{\infty}\bar{G}(t)dt)^{-1}%
=1/E\left[  A\right]  $ as the arrival rate, and $\mu=(\int_{0}^{\infty}%
\bar{F}(t)dt)^{-1}=1/E[V]$ as the service rate. In order to ensure the
existence of the stationary distribution of the system, we require the
following stability condition $\lambda/(c\mu)<1$.

\subsection{Elements of the simulation strategy: upper bound and coupling}

We first introduce some additional notations which we shall use to describe
the upper bound a) in the application of the DCFTP framework. Let
\[
\mathcal{T}^{0}:=\{T_{n}^{0}:n\in\mathbb{Z}\backslash\{0\}\}
\]
be a time-stationary renewal point process with $T_{n}^{0}>0$ if $n\geq1$ and
$T_{-n}<0$ if $n\geq1$ (the $T_{n}^{0}$'s are sorted in a non-decreasing order
in $n$). The time $T_{n}^{0}$ for $n\geq1$ represents the arrival time of the
$n$-th customer into the system after time zero and, for $n\geq1$, $T_{-n}%
^{0}$ is the arrival time of the $n$-th customer, backwards in time, from time
zero. We also define $T_{n}^{0,+}=\inf\{T_{m}^{0}:T_{m}^{0}>T_{n}^{0}\}$, that
is, the arrival time of the next customer after $T_{n}^{0}$. If $n\geq1$ or
$n\leq-2$, $T_{n}^{0,+}=T_{n+1}^{0}$. However, $T_{-1}^{0,+}=T_{1}^{0}$.
Similarly, we write $T_{n}^{0,-}=\sup\{T_{m}^{0}:T_{m}^{0}<T_{n}^{0}\}$.

Define $A_{n}=T_{n}^{0,+}-T_{n}^{0}$ for all $n\in\mathbb{Z}\backslash\{0\}$.
Note that $A_{n}$ is the interarrival time between the customer arriving at
time $T_{n}^{0}$ and the next customer. $A_{n}$ has CDF $G\left(
\cdot\right)  $ for $n\geq1$ or $n\leq-2$, but $A_{-1}$ has a different
distribution due the inspection paradox.

Now, for $i\in\{1,2,...,c\}$ we introduce iid time-stationary renewal point
processes
\[
\mathcal{T}^{i}:=\{T_{n}^{i}:n\in\mathbb{Z}\backslash\{0\}\},
\]
as before we have that $T_{n}^{i}>0$ for $n\geq1$ and $T_{-n}^{i}<0$ if
$n\geq1$ with the $T_{n}^{i}$'s sorted in a non-decreasing order. We also
define $T_{n}^{i,+}=\inf\{T_{m}^{i}:T_{m}^{i}>T_{n}^{i}\}$ and $T_{n}%
^{i,-}=\sup\{T_{m}^{i}:T_{m}^{i}<T_{n}^{i}\}$. Then we let $V_{n}^{i}%
=T_{n}^{i,+}-T_{n}^{i}$. We assume that $V_{n}^{i}$ has CDF $F\left(
\cdot\right)  $ for $n\geq1$ and $n\leq-2$. As we shall explain, the
$V_{n}^{i}$'s are \textit{activities} which are executed by the $i$-th server
in the upper bound process.

Next, we define, for each $i\in\{0,1,...,c\}$, and any $u\in\left(
-\infty,\infty\right)  $, the counting process
\[
N_{u}^{i}\left(  t\right):=\left\vert [u,u+t]\cap\mathcal{T}^{i}\right\vert
,
\]
for $t\geq0$, where
$\vert$
$\vert$
denotes cardinality. Note that as $T_{-1}^{i}<0<T_{1}^{i}$ by stationary,
$N_{0}^{i}\left(  0\right)  =0$. For simplicity in the notation let us write
$N^{i}\left(  t\right)  =N_{0}^{i}\left(  t\right)  $ if $t\geq0$ and
$N^{i}\left(  t\right)  =N_{t}^{i}\left(  -t\right)  $ if $t\leq0$.

The quantity $N_{u}^{0}\left(  t\right)  $ is the number of customers who
arrive during the time interval $[u,u+t]$. In the upper bound process, each of
the $c$ servers performs two types of activities: services and vacations.
$N_{u}^{i}\left(  t\right)  $ is the number of activities initiated by server
$i$ during the time interval $[u,u+t]$.

\subsubsection{The upper bound process}

We shall refer to the upper bound process as the \textit{vacation system}, for
reasons which will become apparent. Let us explain first in words how does the
vacation system operate. Customers arrive to the vacation system according to
$\mathcal{T}^{0}$, and the system operates similarly to a $GI/GI/c$ queue,
except that, every time a server (say server $i^{\ast}$) finishes an
activity (i.e. service or a vacation), if there is no customer waiting to be
served in the queue, server $i^{\ast}$ takes a vacation which has the same
distribution as the service time distribution; if there is at least one
customer waiting, such customer starts to be served by server $i^{\ast}$.
Similar vacation models have been used in \cite{Whitt:1970} and
\cite{GarGol:2013}.

More precisely, let $Q_{v}(t)$ denote the number of people waiting in queue at time $t$
in the stationary vacation system. We write $Q_{v}(t_{-}):=\lim_{s\uparrow
t}Q_{v}\left(  s\right)  $ and $dQ_{v}(t):=Q_{v}(t)-Q_{v}(t_{-})$.
Also, for for any $t\geq0$, $i\in\{0,...,c\}$ and each $u\in(-\infty,\infty)$,
define
\[
N_{u}^{i}(t_{-}):=\lim_{h\uparrow0}N_{u-h}^{i}\left(  t\right)  ,
\]
and let $dN_{u}^{i}\left(  t\right) :=N_{u}^{i}(t)-N_{u}^{i}(t_{-})$ for all
$t\geq0$ (note that as $N_{u}^{i}\left(  0_{-}\right)  =0$,
$dN_{u}^{i}\left(  0\right)  =N_{u}^{i}\left(  0_{-}\right)  $). Similarly,
for $t\leq0$, $N^{i}\left(  t_{-}\right)  =N_{t}^{i}\left(  \left\vert
t\right\vert _{-}\right)  $.

We also introduce $X_{u}(t):=N_{u}^{0}(t)-\sum_{i=1}^{c}N_{u}^{i}(t)$. Then the
dynamics of $\left(  Q_{v}\left(  t\right)  :t>0\right)  $ satisfy%
\begin{equation}
dQ_{v}\left(  t\right)  =dX_{0}\left(  t\right)  +I\left(  Q_{v}\left(
t_{-}\right)  =0\right)  \sum_{i=1}^{c}dN_{0}^{i}(t), \label{SDE_A}%
\end{equation}
given $Q_{v}\left(  0\right)  $. Note that here we are using the fact that
arrivals do not occur at the same time as the start of activity times; this is
because the processes $\mathcal{T}^{i}$ are time stationary (and independent)
renewal processes in continuous time so that $T_{-1}^{i}$ and $T_{1}^{i}$ have
a density.

It follows from standard arguments for Skorokhod mapping \cite{ChenYao:2013} that for $t\geq0$%
\[
Q_{v}(t)=Q_{v}(0)+X_{0}(t)-\inf_{0\leq s\leq t}\left(  \left(  X_{0}%
(s)+Q_{v}(0)\right)  ^{-}\right)  ,
\]
where $\left(  X_{0}(s)+Q_{v}(0)\right)  ^{-}=\min\left(  X_{0}\left(
s\right)  +Q_{v}(0),0\right)  $. Moreover, using Lyons construction we have
that $t\geq0$
\begin{equation}
Q_{v}(-t)=\sup_{s\geq t}X_{-s}\left(  0\right)  -X_{-t}\left(  0\right)
\label{Stat_Backwards}%
\end{equation}
(see, for example, Proposition 1 of \cite{BlanChen:2013}). $(Q_{v}%
(t):t\in\left(  -\infty,\infty\right)  )$ is a well defined process by virtue
of the stability condition $\lambda/(\mu c)<1$.

The vacation system and the target process (the $GI/GI/c$ queue) will be
coupled by using the same arrival stream of customers, $\mathcal{T}^{0}$, and
assuming that each customer brings his own service time. In particular, the
evolution of the underlying $GI/GI/c$ queue is described using a sequence of
the form $\left(  \left(  T_{n}^{0},V_{n}\right)  :n\in\mathbb{Z}%
\backslash\{0\}\right)  $, where $V_{n}$ is the service requirement of the
customer arriving at time $T_{n}^{0}$. The $V_{n}$'s must be extracted from
the evolution of $Q_{v}\left(  \cdot\right)  $ so that the same service times
are matched to the common arrival stream both in the vacation system and in the
target process.

\subsubsection{The coupling: extracting service times for each
costumer\label{Subsec_Coupling}}

In order to match the service times corresponding to each of the arriving
customers in the vacation system we define the following auxiliary processes.
For every $i\in\{1,...,c\}$, any $t>0$, and each $u\in\left(  -\infty
,\infty\right)  $, let $\sigma_{u}^{i}\left(  t\right)  $ denote the number of
service initiations by server $i$ during the time interval $[u,u+t]$. Observe
that
\[
\sigma_{u}^{i}\left(  t\right)  =\int_{[u,u+t]}I\left(  Q_{v}\left(
s_{-}\right)  >0\right)  dN_{u}^{i}\left(  s\right)  .
\]
That is, we count service initiations which start at time $T_{k}^{i}\in\lbrack
u,u+t]$ if and only if $Q_{v}\left(  T_{k-}^{i}\right)  >0$. Once again, here
we use that arrival times and activity initiation times do not occur simultaneously.

We now explain how to match the service time of the customer arriving at
$T_{n}^{0}$. First, such customer occupies position $Q_{v}\left(  T_{n}%
^{0}\right)  \geq1$ when he enters the queue. Let $D_{n}^{0}$ be the delay (or
waiting time) inside the queue of the customer arriving at
$T_{n}^{0}$, then we have that%
\[
D_{n}^{0}=\inf\{t\geq0:Q_{v}\left(  T_{n}^{0}\right)  =\sum_{i=1}^{c}%
\sigma_{T_{n}^{0}}^{i}\left(  t\right)  \},
\]
and therefore,
\begin{equation}
V_{n}=\sum_{i=1}^{c}V_{N^{i}\left(  T_{n}^{0}+D_{n}^{0}\right)  }^{i}\cdot
dN^{i}\left(  T_{n}^{0}+D_{n}^{0}\right)  . \label{EQ_V_MATCH}%
\end{equation}
Observe that the previous equation is valid because there is a unique
$i\left(  n\right)  \in\{1,...,c\}$ for which $dN^{i\left(  n\right)  }\left(
T_{n}^{0}+D_{n}^{0}\right)  =1$ and $dN^{j}\left(  T_{n}^{0}+D_{n}^{0}\right)
=0$ if $j\neq i\left(  n\right)  $ (ties are not possible because of the time
stationarity of the $\mathcal{T}^{i}$s), so we obtain that (\ref{EQ_V_MATCH})
is equivalent to
\[
V_{n}=V_{N^{i\left(  n\right)  }\left(  T_{n}^{0}+D_{n}^{0}\right)
}^{i\left(  n\right)  }.
\]
We shall explain in the appendix to this section,
that $\left(  V_{n}:n\in\mathbb{Z}\backslash\{0\}\right)  $ and $\left(
T_{n}^{0}:n\in\mathbb{Z}\backslash\{0\}\right)  $ are two independent
sequences and the $V_{n}$'s are iid copies of $V$.

\subsection{A family of $GI/GI/c$ queues and the target $GI/GI/c$ stationary
system}

We now describe the evolution of a family of standard $GI/GI/c$ queues. Once
we have the sequence $\left(  \left(  T_{n}^{0},V_{n}\right)  :n\in
\mathbb{Z}\backslash\{0\}\right)  $ we can proceed to construct a family of
continuous-time Markov processes $\left(  Z_{u}(t;z):t\geq0\right)  $ for each
$u\in\left(  -\infty,\infty\right)  $, given the initial condition
$Z_{u}\left(  0;z\right)  =z$. We write $z=\left(  q,r,e\right)  $, and set%
\[
Z_{u}(t;z):=\left(  Q_{u}\left(  t;z\right)  ,R_{u}\left(  t;z\right)
,E_{u}\left(  t;z\right)  \right)  ,
\]
for $t\geq0$, where $Q_{u}\left(  t;z\right)  $ is the number of people in the
queue at time $u+t$ ($Q_{u}\left(  0;z\right)  =q$), $R_{u}(t;z)$ is the
vector of ordered (ascending) remaining service times of the $c$ servers at
$u+t$ ($R_{u}(0;z)=r$), and $E_{u}(t;z)$ is the time elapsed since the
previous arrival at $u+t$ ($E_{u}(0;z)=e$).

We shall always use $E_{u}(0;z)=e=u-\sup\{T_{n}^{0}:T_{n}^{0}\leq u\}$ and we
shall select $q$ and $r$ appropriately based on the upper bound. The evolution
of the process $\left(  Z_{u}\left(  s;z\right)  :0<s\leq t\right)  $ is
obtained by feeding the traffic $\{\left(  T_{n}^{0},V_{n}\right)
:u<T_{n}^{0}\leq u+s\}$ for $s\in(0,t]$ into a FCFS $GI/GI/c$ queue with
initial conditions given by $z$. Constructing $\left(  Z_{u}\left(
s;z\right)  :0<s\leq t\right)  $ using the traffic trace $\{\left(  T_{n}%
^{0},V_{n}\right)  :u<T_{n}^{0}\leq u+s\}$ for $s\in(0,t]$ is standard (see
for example Chapter 3 of \cite{RubKro:2011}).

One can further describe the evolution of the underlying $GI/GI/c$ at arrival
epochs, using the Kiefer-Wolfowitz vector \cite{Asm:2003}. In particular, for
every non-negative vector $w\in\mathbb{R}^{c}$, such that $w^{\left(
i\right)  }\leq w^{\left(  i+1\right)  }$ (where $w^{\left(  i\right)  }$ is
the $i$-th entry of $w$), and each $k\in\mathbb{Z}\mathbf{\backslash}\{0\}$
the family of processes $\{W_{k}\left(  T_{n}^{0};w\right)  :n\geq
k,n\in\mathbb{Z}\mathbf{\backslash}\{0\}\}$ satisfies
\begin{align}
W_{k}\left(  T_{n}^{0,+};w\right)   &  =\mathcal{S}\left(  \left(
W_{k}\left(  T_{n}^{0};w\right)  +V_{n}\mathbf{e}_{1}-A_{n}\mathbf{1}\right)
^{+}\right)  ,\label{K_W_GGc}\\
W_{k}\left(  T_{k}^{0};w\right)   &  =w.\nonumber
\end{align}
where $\mathbf{e}_{1}=\left(  0,0,...,1\right)  ^{T}\in\mathbb{R}^{c}$,
$\mathbf{1}=\left(  1,...,1\right)  ^{T}\in\mathbb{R}^{c}$, and $\mathcal{S}$
is the sorting operator which arranges the entries in a vector in ascending
order. In simple words, $W_{k}\left(  T_{n}^{0};w\right)  $ for $k\geq1$
describes the Kiefer-Wolfowitz vector as observed by the customer arriving at
$T_{n}^{0}$, assuming that customer who arrived at $T_{k}^{0}$, $k\leq n$,
experienced the Kiefer-Wolfowitz state $w$.

Recall that the first entry of $W_{k}\left(  T_{n}^{0},w\right)  $, namely
$W_{k}^{\left(  1\right)  }\left(  T_{n}^{0},w\right)  $, is the waiting time
of the customer arriving at $T_{n}^{0}$ (given the initial condition $w$ at
$T_{k}^{0}$). More generally, the $i$-th entry of $W_{k}\left(  T_{n}%
^{0};w\right)  $, namely, $W_{k}^{\left(  i\right)  }\left(  T_{n}%
^{0};w\right)  $, is the virtual waiting time of the customer arriving at
$T_{n}^{0}$ if he decided to enter service immediately after there are at
least $i$ servers free once he reaches the head of the line. In other words,
one can also interpret $W_{k}\left(  T_{n}^{0};w\right)  $ as the remaining
vector of workloads (sorted in ascending order) that would be processed by
each of the $c$ servers at $T_{n}^{0}$, if no more arrivals got into the
system after time $T_{n}^{0}$.

We are now ready to construct the stationary version of the $GI/GI/c$ queue.
Namely, for each $n\in\mathbb{Z}\backslash\{0\}$ and every $t\in\left(
-\infty,\infty\right)  $ we define $W\left(  n\right)  $ and $Z\left(
t\right)  $ via
\begin{align}
W\left(  n\right)   & :=\lim_{k\rightarrow-\infty}W_{k}\left(  T_{n}%
^{0};0\right)  ,\label{Limits_W_Z_GGc}\\
Z\left(  t\right)   &  :=\left(  Q\left(  t\right)  ,R\left(  t\right)
,E\left(  t\right)  \right)  =\lim_{u\rightarrow-\infty}Z_{u}\left(
t-u,z_{-}\right)  ,\nonumber
\end{align}
where $z_{-}=\left(  0,0,e\right)  $.

We shall show in Proposition \ref{Thm_Const_St} that these limits are well defined.

\subsection{The analogue of the Kiefer-Wolfowitz process for the upper bound
system}

In order to complete the coupling strategy we also describe the evolution of
the analog Kiefer-Wolfowitz vector induced by the vacation system, which we
denote by $\left(  W_{v}\left(  T_{n}^{0}\right)  :n\in\mathbb{Z}%
\backslash\{0\}\right)  $, where $v$ stands for vacation. As with the $i$-th
entry of the Kiefer-Wolfowitz vector of a $GI/GI/c$ queue, the $i$-th entry of
$W_{v}\left(  T_{n}^{0}\right)  $, namely $W_{v}^{\left(  i\right)  }\left(
T_{n}^{0}\right)  $, is the virtual waiting time of the customer arriving at
time $T_{n}^{0}$ if he decided to enter service immediately after there are at
least $i$ servers free once he reaches the head of the line (assuming that
servers become idle once they see, after the completion of a current activity,
the customer in question waiting in the head of the line).

To describe the Kiefer-Wolfowitz vector induced by the vacation system
precisely, let $U^{i}\left(  t\right)  $ be the time until the next renewal
after time $t$ in $\mathcal{T}^{i}$, that is $U^{i}\left(  t\right)
=\inf\{T_{n}^{i}:T_{n}^{i}>t\}-t$. So, for example, $U^{0}\left(  T_{n}%
^{0}\right)  =A_{n}$ for $n\geq1$. Let $U\left(  t\right)  =\left(
U^{1}\left(  t\right)  ,...,U^{c}\left(  t\right)  \right)  ^{T}$. We then
have that%
\begin{equation}
W_{v}\left(  T_{n}^{0}\right)  =D_{n}^{0}\mathbf{1+}\mathcal{S}\left(
U\left(  \left(  T_{n}^{0}+D_{n}^{0}\right)  _{-}\right)  \right).
\label{Gen_Def_W_v}%
\end{equation}
In particular, note that $W_{v}^{\left(  1\right)  }\left(  T_{n}^{0}\right)
=D_{n}^{0}$.

Actually, in order to draw a closer connection to the Kiefer-Wolfowitz vector
recursion of a standard $GI/GI/c$ queue, let us write
\[
\mathcal{S}\left(  U\left(  \left(  T_{n}^{0}+D_{n}^{0}\right)  _{-}\right)
\right)  =(U^{\left(  1\right)  }\left(  \left(  T_{n}^{0}+D_{n}^{0}\right)
_{-}\right)  ,...,U^{\left(  c\right)  }\left(  \left(  T_{n}^{0}+D_{n}%
^{0}\right)  _{-}\right)  ^{T},
\]
and suppose that $U^{\left(  i\right)  }\left(  \left(  T_{n}^{0}+D_{n}%
^{0}\right)  _{-}\right)  =U^{j_{i}\left(  n\right)  }\left(  \left(
T_{n}^{0}+D_{n}^{0}\right)  _{-}\right)  \ $(i.e. $j_{i}\left(  n\right)  $ is
the server whose remaining activity time right before $T_{n}^{0}+D_{n}^{0}$ is
the $i$-th smallest in order). Then, define
\[
\bar{W}_{v}\left(  T_{n}^{0}\right)  =W_{v}\left(  T_{n}^{0}\right)
+V_{n}\mathbf{e}_{1}-A_{n}\mathbf{1,}%
\]
and let $\bar{W}_{v}^{\left(  i\right)  }\left(  T_{n}^{0}\right)  $ to be the
$i$-th entry of $\bar{W}_{v}\left(  T_{n}^{0}\right)  $. It is not difficult
to see from the definition of $W_{v}\left(  T_{n}^{0}\right)  $ that%
\[
W_{v}^{\left(  i\right)  }\left(  T_{n}^{0,+}\right)  =\mathcal{S}\left(
\left(  \bar{W}_{v}^{\left(  i\right)  }\left(  T_{n}^{0}\right)  \right)
^{+}+\Xi_{n}^{\left(  i\right)  }\right)  ,
\]
where
\begin{align*}
\Xi_{n}^{\left(  i\right)  }  &  =I(\bar{W}_{v}^{\left(  i\right)  }\left(
T_{n}^{0}\right)  <0)\cdot U^{j_{i}\left(  n\right)  }\left(  \left(
T_{n}^{0,+}+D_{n}^{0,+}\right)  _{-}\right)  ,\text{ and}\\
D_{n}^{0,+}  &  =\inf\{D_{k}^{0}:D_{k}^{0}>D_{n}^{0}\}.
\end{align*}
So, (\ref{Gen_Def_W_v}) actually satisfies%
\begin{equation}
W_{v}\left(  T_{n}^{0,+}\right)  =\mathcal{S}\left(  \left(  W_{v}\left(
T_{n}^{0}\right)  +V_{n}\mathbf{e}_{1}-A_{n}\mathbf{1}\right)  ^{+}%
\mathbf{+}\Xi_{n}\right)  , \label{Recursion_Wv}%
\end{equation}
where $\Xi_{n}=\left(  \Xi_{n}^{\left(  1\right)  },...,\Xi_{n}^{\left(
c\right)  }\right)  $.

\subsection{Description of simulation strategy and main result}

We now describe how the variation of DCFTP that we described in the
Introduction, using the monotonicity of the multiserver queue, and elements a)
to d) apply to our setting.

The upper bound is initialized using the Kiefer-Wolfowitz process associated
to the vacation system. For the lower bound, we shall simply pick the null vector.

The strategy combines the following facts (which we shall discuss in the sequel).

\begin{itemize}
\item \textbf{Fact I:} We can simulate $\sup_{s\geq t}X_{-s}\left(  0\right)
$, $\left(  N_{-t}^{i}\left(  0\right)  :1\leq i\leq c\right)  $, $\left(
N_{0}^{i}\left(  t\right)  :1\leq i\leq c\right)  $, jointly for any given
$t\geq0$. This part, which corresponds to item c), is executed using an
algorithm from \cite{BlanWall:2014} designed to sample the infinite horizon
running time maximum of a random walk with negative drift. We shall explain in
Section \ref{sec:algorithm} how the algorithm in \cite{BlanWall:2014} can be
easily modified to sample $\sup_{s\geq t}X_{-s}\left(  0\right)  $ jointly
with $(N_{-s}^{i}\left(  0\right)  :s\geq0)_{i=0}^{c}$.

\item \textbf{Fact II:} For all $k\leq-1$ and every $k\leq n\leq-1$ we have
that%
\[
W_{k}\left(  T_{n}^{0};0\right)  \leq W\left(  n\right)  \leq W_{k}\left(
T_{n}^{0};W_{v}\left(  T_{k}^{0}\right)  \right)  .
\]
This portion exploits the upper bound a) (i.e. $W_{v}\left(  T_{k}^{0}\right)
$), and the lower bound b) (i.e. 0).

\item \textbf{Fact III:} We can detect that coalescence occurs at some time
$-T\in\lbrack T_{\kappa}^{0},0]$ for some $\kappa\leq-1$ by finding
$n\in\mathbb{Z}_{-}$, $n\geq\kappa$, such that $T_{n}^{0}+W_{\kappa}^{\left(
1\right)  }\left(  T_{n}^{0};W_{v}\left(  T_{\kappa}^{0}\right)  \right)
\leq0$ and%
\[
W_{\kappa}\left(  T_{n}^{0};W_{v}\left(  T_{\kappa}^{0}\right)  \right)
=W_{\kappa}\left(  T_{n}^{0};0\right)  .
\]
This portion is precisely the coalescence detection strategy which uses
monotonicity of the Kiefer-Wolfowitz vector.

\item \textbf{Fact IV}: We can combine Facts I-III to conclude that
\begin{equation}
Z_{T_{\kappa}^{0}}\left(  \left\vert T_{\kappa}^{0}\right\vert ;Q\left(
T_{-\kappa}^{0}\right)  ,\mathcal{S}\left(  U\left(  T_{-\kappa}^{0}\right)
\right)  ,0\right)  =Z\left(  0\right)  \label{Stationary_sample}%
\end{equation}
is stationary. And we also have that%
\[
W_{\kappa}\left(  T_{1}^{0};0\right)  =W\left(  1\right)
\]
follows the stationary distribution of the Kiefer-Wolfowitz vector of a
$GI/GI/c$ queue.
\end{itemize}

The main result of this paper is the following.

\begin{theorem}
\label{Thm_Main_1}Assume (A1) is in force, with $\lambda/(\mu c)\in(0,1)$.
Then Facts I-IV hold true and (\ref{Stationary_sample}) is a stationary sample
of the state of the multi-server system. we can detect coalescence at a time
$-T<0$ such that $E\left(  T\right)  <\infty$.
\end{theorem}

The rest of the paper is dedicated to the proof of Theorem 1. In Section
\ref{sec:vacation} we verify a number of monotonicity properties which in
particular allows us to conclude that the construction of $W\left(  n\right)
$ and $Z\left(  t\right)  $ is legitimate (i.e. that the limits exist almost
surely). This monotonicity properties also yield Fact II and pave the way to
verify Fact III. Section \ref{sec:coalescence} proves the finite expectation
of the coalescence time. In Section \ref{sec:algorithm} we give more details
as how to carry out the simulation of the upper bound process. But before
going over Facts, we conclude this section, as we promised, arguing
that the $V_{n}$'s are iid and independent of the arrival sequence
$\mathcal{T}^{0}$.

\subsection{Appendix: The iid property of the coupled service times and
independence of the arrival process}

In order to explain why the $V_{n}$'s form an iid sequence, independent of the
sequence $\mathcal{T}^{0}=\{T_{n}^{0}:n\in\mathbb{Z}\backslash\{0\}\}$, it is
useful to keep in mind the diagram depicted in Figure \ref{fig:tetris_B},
which illustrates a case involving two servers, $c=2$. \begin{figure}[ptb]
\centering
{\includegraphics[width=.3\textwidth]{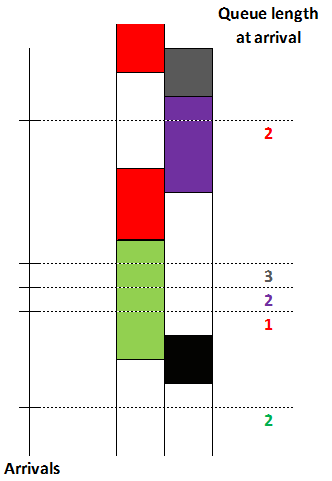} }%
\caption{Matching Procedure of Service Times to Arrival Process}%
\label{fig:tetris_B}%
\end{figure}

The assignment of the service times, as we shall explain, can be thought of as
a procedure similar to a tetris game. Arrival times are depicted by dotted
horizontal lines which go from left to right starting at the left most
vertical line, which is labeled \textquotedblleft Arrivals". Think of the time
line going, vertically, from the bottom of the graph (past) to the top of the
graph (future).

In the right most column in Figure \ref{fig:tetris_B} we indicate the queue
length, right at the time of a depicted arrival (and thus, including the
arrival itself). So, for example, the first arrival depicted in Figure
\ref{fig:tetris_B} observes one customer waiting and thus, including the
arrival himself, there are two customers waiting in queue.

The tetris configuration observed by an arrival at time $T$ is comprised of
two parts: i) the receding horizon, which corresponds to the remaining
incomplete blocks, and ii) the landscape, comprised of the configuration of
complete blocks. So, for example, the first arrival in Figure
\ref{fig:tetris_B} observes a receding horizon corresponding to the two white
remaining blocks which start from the dotted line at the bottom. The landscape
can be parameterized by a sequence of block sizes and the order of the
sequence is given by the way in which the complete blocks appear from bottom
to top -- this is precisely the tetris-game assignment. There are no ties
because of the continuous time stationarity and independence of the underlying
renewal processes. The colors are, for the moment, not part of the landscape.
We will explain the meaning of the colors momentarily.

The assignment of the service times is done as follows. The arriving customer
reads off the right-most column (with heading "Queue length at arrival") and
selects the block size labeled precisely with the number indicated by the
"Queue length at arrival". So, there are two distinctive quantities to keep in
mind assigned to each player (i.e. arriving customer): a) the landscape (or
landscape sequence, which, as indicated can be used to reconstruct the
landscape), and b)\ the \textit{service time}, which is the complete block
size occupying the "Queue length at arrival"-th position in the landscape sequence.

The color code in Figure \ref{fig:tetris_B} simply illustrates quantity b) for
each of the arrivals. So, for example, the first arrival, who reads "Queue
length at arrival = 2" (which we have written in green color),\ gets assigned
the second complete block, which we have depicted in green. Similarly, the
second arrival depicted, reads off the number "1" (written in red) and gets
assigned the first red block depicted (from bottom to top). The very first
complete block (from bottom to top), which is depicted in black, corresponds
to the service time assigned to the customer ahead of the customer who
collected the green block. The number "1" (in red) is obtained by observed
that the customer with the initial black block has departed.

Now we argue the following properties:

1) The service times are iid copies of $V$.

2)\ The service times are independent of $\mathcal{T}^{0}$.

About property 1): The player arriving at time $T$, reads a number,
corresponding to the queue length, which is obtained by the \textit{past
filtration} $\mathcal{F}_{T}$ generated by $\cup_{k\in\mathbb{Z}%
\backslash\{0\},0\leq i\leq c}\{T_{k}^{i}:T_{k}^{i}\leq T\}$. Conditional on
the receding horizon (i.e. remaining incomplete block sizes), $\mathcal{R}%
_{T}$, the past filtration is independent of the landscape. This is simply the
Markov property applied to the forward residual life time process of each of
the $c$ renewal processes represented by the $c$ middle columns. Moreover,
conditional on $\mathcal{R}_{T}$, each landscape forms a sequence of iid
copies of $V$ because of the structure of the underlying $c$ renewal processes
corresponding to the middle columns. So, let $Q\left(  T\right) $ denote the
queue length $T$ (including the arrival at time $T$), which is a function of
the past filtration, and let $\{L_{T}\left(  k\right)  :k\geq1\}$ be the
landscape sequence observed at time $T$, so that $L_{T}\left(  Q\left(
T\right)  \right)  $ is the service time of the customer who arrives at time
$T$. We then have that for any positive and bounded continuous function
$f\left(  \cdot\right)  $
\[
E[f\left(  L_{T}\left(  Q\left(  T\right)  \right)  \right)  |\mathcal{R}%
_{T}]=E[f\left(  L_{T}\left(  1\right)  \right)  |\mathcal{R}_{T}]=E[f\left(
V\right)  |\mathcal{R}_{T}],
\]
precisely because conditional on $\mathcal{R}_{T}$, $Q\left(  T\right)  $
(being $\mathcal{F}_{T}$ measurable) is independent of $L_{T}$.

To verify the iid property, let $f_{1},f_{2}$ be non-negative and bounded
continuous functions. And assume that $T_{1}<T_{2}$ are arrival times in
$\mathcal{T}^{0}$ (not necessarily consecutive). Then,
\begin{align*}
&  E[f_{1}\left(  L_{T_{1}}\left(  Q\left(  T_{1}\right)  \right)  \right)
f_{2}\left(  L_{T_{2}}\left(  Q\left(  T_{2}\right)  \right)  \right)  ]\\
&  =E[E[f_{1}\left(  L_{T_{1}}\left(  Q\left(  T_{1}\right)  \right)  \right)
f_{2}(L_{T_{2}}\left(  Q\left(  T_{2}\right)  \right)  )|\mathcal{F}_{T_{2}%
},\mathcal{R}_{T_{2}}]]\\
&  =E[f_{1}\left(  L_{T_{1}}\left(  Q\left(  T_{1}\right)  \right)  \right)
E[f_{2}(L_{T_{2}}\left(  Q\left(  T_{2}\right)  \right)  )|\mathcal{F}_{T_{2}%
},\mathcal{R}_{T_{2}}]]\\
&  =E[f_{1}(L_{T_{1}}\left(  Q\left(  T_{1}\right)  \right)  )]E[f_{2}%
(V)]=E[f_{1}(V)]E[f_{2}(V)].
\end{align*}
The same argument extends to any subset of arrival times and thus the iid
property follows.

About property 2): Note that in the calculations involving property 1), the
actual values of the arrival times $T$, and $T_{1}$ and $T_{2}$ are
irrelevant. The iid property of the service times is established path-by-path
conditional on the observed realization $\mathcal{T}^{0}$. Thus the
independence of the arrival process and service times follows immediately.


\section{Monotonicity Properties and Stationary $GI/GI/c$ System
\label{sec:vacation}}

This section we will present several lemmas which contain useful monotonicity
properties. The proofs of the lemmas are given at the end of this section in
order to quickly arrive to the main point of this section, which is the
construction of a stationary version of the $GI/GI/c$ queue.

First we recall that the Kiefer-Wolfowitz vector of $GI/GI/c$ queue is
monotone in the initial condition (\ref{Mon_KW_IC}) and invoke a property
(\ref{MON_LYONS}) which will allows us to construct a stationary version of
the Kiefer-Wolfowitz vector of our underlying $GI/GI/c$ queue, using Lyons construction.

\begin{lemma}
\label{Lem_Mon_KW}For $n\geq k,$ $k,n\in\mathbb{Z}\backslash\{0\}$,
$w^{+}>w^{-}$,
\begin{equation}
W_{k}\left(  T_{n}^{0};w^{+}\right)  \geq W_{k}\left(  T_{n}^{0};w^{-}\right)
. \label{Mon_KW_IC}%
\end{equation}
Moreover, if $k\leq k^{\prime}\leq n$%
\begin{equation}
W_{k}\left(  T_{n}^{0};0\right)  \geq W_{k^{\prime}}\left(  T_{n}%
^{0};0\right)  . \label{MON_LYONS}%
\end{equation}

\end{lemma}

The second result allows to make precise a sense in which the vacation system
dominates a suitable family of $GI/GI/c$ systems, in terms of the underlying
Kiefer-Wolfowitz vectors.

\begin{lemma}
\label{lm:dome} For$~~n\geq k,$ $k,n\in\mathbb{Z}\backslash\{0\}$,
\[
W_{v}\left(  T_{n}^{0}\right)  \geq W_{k}\left(T_n^0;W_{v}\left(  T_{k}%
^{0}\right)  \right)  .
\]

\end{lemma}

The next result shows that in terms of queue length processes, the vacation
system also dominates a family of $GI/GI/c$ queue, which we shall use to
construct the upper bounds.

\begin{lemma}
\label{th:dom1} Let $q=Q_{v}(u)$, $r=\mathcal{S}\left(  U\left(  u\right)
\right)  $, and $e=u-\sup\{T_{n}^{0}:T_{n}^{0}\leq u\}$, so that
$z^{+}=\left(  q,r,e\right)  $ and $z^{-}=\left(  0,0,e\right)  $ then for
$t\geq u$
\[
Q_{u}(t-u;z^{-})\leq Q_{u}(t-u;z^{+})\leq Q_{v}(t).
\]

\end{lemma}

Using Lemmas \ref{Lem_Mon_KW}, \ref{lm:dome}, and \ref{th:dom1} we can
establish the following result.

\begin{proposition}
\label{Thm_Const_St}The limits defining $W\left(  n\right)  $ and $Z\left(
t\right)  $ in (\ref{Limits_W_Z_GGc}) exist almost surely. Moreover, we have
that Fact II holds.
\end{proposition}

\begin{proof}
[Proof of Proposition \ref{Thm_Const_St}]Using Lemma \ref{lm:dome} and Lemma
\ref{Lem_Mon_KW} we have that
\[
W_{v}\left(  T_{n}^{0}\right)  \geq W_{k}\left(  T_{n}^{0};W_{v}\left(
T_{k}^{0}\right)  \right)  \geq W_{k}\left(  T_{n}^{0};0\right)  .
\]
So, by property (\ref{MON_LYONS}) in Lemma \ref{Lem_Mon_KW} we conclude that
the limit defining $W\left(  n\right)  $ exists almost surely and that
\begin{equation}
W\left(  n\right)  \leq W_{v}\left(  T_{n}^{0}\right)  . \label{Bound_KF_KF}%
\end{equation}
Similarly, using Lemma \ref{th:dom1} we can obtain the existence of the limit $Q(t)$ and we have that
$Q\left(  t\right)  \leq Q_{v}\left(  t\right)$. Moreover, by
convergence of the Kiefer-Wolfowitz vectors we obtain the $i$-th entry of
$R\left(  T_{n}^{0}+W^{(1)}(n)\right)  $, namely, $R^{\left(  i\right)  }\left(
T_{n}^{0}+W^{(1)}(n)\right)  =\left(  W^{\left(  i\right)  }\left(  n\right)  -W^{\left(
1\right)  }\left(  n\right)  \right)  ^{+}$,
where $i\in\{1,...,c\}$. Clearly,
since the age process has been taken underlying $\mathcal{T}^{0}$, we have
that $E\left(  t\right)  =t-\sup\{T_{n}^{0}:T_{n}^{0}\leq t\}$. The fact that
the limits are stationary follows directly from the limiting procedure and it
is standard in Lyons-type constructions. For Fact II, we use the identity
$W\left(  n\right)  =W_{k}\left(  T_{n}^{0};W\left(  k\right)  \right)  $,
combined with Lemma \ref{Lem_Mon_KW} to obtain
\[
W_{k}\left(  T_{n}^{0};0\right)  \leq W_{k}\left(  T_{n}^{0};W\left(
k\right)  \right)  =W\left(  n\right)  ,
\]
and then we apply Lemma \ref{lm:dome}, together with (\ref{Bound_KF_KF}), to
obtain%
\[
W\left(  n\right)  = W_{k}\left(  T_{n}^{0};W\left(  k\right)  \right)  \leq
W_{k}\left(  T_{n}^{0};W_{v}\left(  T_{k}^{0}\right)  \right)  .
\]
The previous two inequalities are precisely the statement of Fact II.
\end{proof}

\subsection{Proof of technical Lemmas}

\begin{proof}
[Proof of Lemma \ref{Lem_Mon_KW}]Both facts are standard, the first one can be
easily shown using induction. Specifically, we first notice that $W_{k}%
(T_{k}^{0}; w^{+})=w^{+}>w^{-}=W_{k}(T_{n}^{0};w^{-})$ Suppose that
$W_{k}(T_{n}^{0}; w^{+})\geq W_{k}(T_{n}^{0};w^{-})$ for some $n \geq0$, then
\begin{align*}
W_{k}(T_{n+1}^{0}; w^{+})  &  =\mathcal{S}\left(  \left(  W_{k}(T_{n}^{0};
w^{+})+V_{n}-A_{n}\right)  ^{+}\right) \\
&  \geq\mathcal{S}\left(  \left(  W_{k}(T_{n}^{0}; w^{-})+V_{n}-A_{n}\right)
^{+}\right)  =W_{k}(T_{n+1}^{0}; w^{-}).
\end{align*}
For inequality (\ref{MON_LYONS}), we note that
$W_{k}\left(  T_{k^{\prime}}^{0};0\right)  \geq W_{k^{\prime}}\left(
T_{k^{\prime}}^{0};0\right)  =0$,
and therefore, due to (\ref{Mon_KW_IC}), we have that%
\[
W_{k}\left(  T_{n}^{0};0\right)  =W_{k^{\prime}}\left(  T_{n}^{0};W_{k}\left(
k^{\prime};0\right)  \right)  \geq W_{k^{\prime}}\left(T_n^0;0\right)  .
\]
\end{proof}

\bigskip

\begin{proof}
[Proof of Lemma \ref{lm:dome}]This fact follows immediately by induction from
equations (\ref{K_W_GGc}) and (\ref{Recursion_Wv}) using the fact that
$\Xi_{n}\geq0$.
\end{proof}

\bigskip

\begin{proof}
[Proor of Lemma \ref{th:dom1}] We first prove the inequality $Q_{u}(t-u;z^{+})\leq
Q_{v}(t)$. Note that $U^{i}\left(  u\right)  >0$ for all $u$ (the forward
residual life time process is right continuous), so the initial condition $r$
indicates that all the servers are busy (operating) and the initial $q\geq0$
customers will leave the queue (i.e. enter service) at the same time
in the vacation system as under the evolution of $Z_{u}\left(
\cdot;z^{+}\right)  $. Now, let us write $N=\inf\{n:T_{n}^{0}\geq u\}$ (in
words, the next arriving customer at or after $u$ arrives at time $T_{N}^{0}%
$). It is easy to see that $\mathcal{S}\left(
U\left(  T_{N}^{0}\right)  \right)  \geq R_{u}(T_{N}^{0}-u;z^{+})$; to wit, if
$T_{N}^{0}$ occurs before any of the servers becomes idle, then we have
equality, and if $T_{N}^{0}$ occurs after, say $l\geq1$, servers become idle,
then $R_{u}(T_{N}^{0}-u;z^{+})$ will have $l$ zeroes and the bottom $c-l$
entries will coincide with those of $\mathcal{S}\left(  U\left(  T_{N}%
^{0}\right)  \right)  $, which has strictly positive entries. So, if $w_{N}$
is the Kiefer-Wolfowitz vector observed by the customer arriving at $T_{N}%
^{0}$ (induced by $Q_{u}(\cdot-u;z^{+})$), then we have $W_{v}\left(
T_{N}^{0}\right)  \geq w_{N}$. By monotonicity of the Kiefer-Wolfowitz vector
in the initial condition and because of Lemma \ref{lm:dome}, we have
\[
W_{v}\left(  T_{k}^{0}\right)  \geq W_{N}\left(  T_{k}^{0},W_{v}\left(
T_{N}^{0}\right)  \right)  \geq W_{N}\left(  T_{k}^{0},w_{N}\right)  ,
\]
for all $k\geq N$, and hence, $T_{k}^{0}+D_{k}^{0}\geq
T_{k}^{0}+W_{N}^{\left(  1\right)  }\left(  T_{k}^{0},w_{N}\right)  $.
Therefore, the departure time from the queue (i.e. initiation of service) of
the customer arriving at $T_{k}^{0}$ in the vacation system occurs after the
departure time from the queue of the customer arriving at time $T_{k}^{0}$ in
the $GI/GI/c$ queue. Consequently, we conclude that the set of customers
waiting in the queue in the $GI/GI/c$ system at time $t$ is a subset of the
set of customers waiting in the queue in the vacation system at the same time.
Similarly, we consider $Q_{u}(t-u;z^{-})\leq Q_{u}(t-u;z^{+})$, which is
easier to establish, since for $k\geq N$ (with the earlier definition of $T_{N}^{0}$
and $w_{N}$),
\[
W_{N}\left(  T_{k}^{0};w_{N}\right)  \geq W_{N}\left(  T_{k}^{0};0\right)  ,
\]
So the set of customers waiting in the queue in the lower bound $GI/GI/c$
system at time $t$ is a subset of the set of customers waiting in the upper
bound $GI/GI/c$ system at the same time.
\end{proof}


\section{The coalescence detection in finite time \label{sec:coalescence}}

In this section, we give more details about the coalescence detection scheme.
The next result corresponds to Fact III and Fact IV.

\begin{proposition}
\label{Lem_Coupling_CT}Suppose that $w^{+}=W_{v}\left(  T_{k}^{0}\right)  $
and $w^{-}=0$. Assume that $W_{k}\left(  T_{n}^{0};w^{+}\right)  =W_{k}\left(
T_{n}^{0};w^{-}\right)  $ for some $k\leq n\leq-1$. Then, $W_{k}\left(
T_{m}^{0};w^{+}\right)  =W\left(  m\right)  =W_{k}\left(  T_{m}^{0}%
;w^{-}\right)  $ for all $m\geq n$. Moreover, for all $t\geq T_{n}^{0}%
+W_{k}^{\left(  1\right)  }\left(  T_{n}^{0};w^{+}\right)  $,%
\begin{equation}
Z_{T_{k}^{0}}(t-T_{k}^{0};Q_{v}\left(  T_{k}^{0}\right)  ,\mathcal{S}\left(
U\left(  T_{k}^{0}\right)  \right)  ,0)=Z_{T_{k}^{0}}(t-T_{k}^{0}%
;0,0,0)=Z\left(  t\right)  . \label{EQ_Z}%
\end{equation}

\end{proposition}

\begin{proof}
[Proof of Proposition \ref{Lem_Coupling_CT}]The fact that
\[
W_{k}\left(  T_{m}^{0};w^{+}\right)  =W\left(  m\right)  =W_{k}\left(
T_{m}^{0};w^{-}\right)
\]
for $m\geq n$ follows immediately from the recursion defining the
Kiefer-Wolfowitz vector. Now, to show the first equality in (\ref{EQ_Z}) it
suffices to consider $t=T_{n}^{0}+W_{k}^{\left(  1\right)  }\left(  T_{n}%
^{0};w^{+}\right)  $, since from $t\geq T_{n}^{0}$ the input is exactly the
same and everyone coming after $T_{n}^{0}$ will depart the queue and enter
service after $T_{n}^{0}+W_{k}^{\left(  1\right)  }\left(  T_{n}^{0}%
;w^{+}\right)  $. The arrival processes (i.e. $E_{u}\left(  \cdot\right)  $)
clearly agree, so we just need to verify that the queue lengths and the
residual service times agree. First, note that
\begin{align}
&  R_{T_{k}^{0}}(T_{n}^{0}+W_{k}^{\left(  1\right)  }\left(  T_{n}^{0}%
;w^{+}\right)  -T_{k}^{0};Q_{v}\left(  T_{k}^{0}\right)  ,\mathcal{S}\left(
U\left(  T_{k}^{0}\right)  \right)  ,0)\label{EQ_DIS_R}\\
&  =W_{k}\left(  T_{n}^{0};w^{+}\right)  -W_{k}^{\left(  1\right)  }\left(
T_{n}^{0};w^{+}\right)  \mathbf{1}\nonumber\\
&  =W_{k}\left(  T_{n}^{0};w^{-}\right)  -W_{k}^{\left(  1\right)  }\left(
T_{n}^{0};w^{-}\right)  \mathbf{1}\nonumber\\
&  =R_{T_{k}^{0}}(T_{n}^{0}+W_{k}^{\left(  1\right)  }\left(  T_{n}^{0}%
;w^{-}\right)  -T_{k}^{0};0,0,0).\nonumber
\end{align}
So, the residual service times of both upper and lower bound processes agree.
The agreement of the queue lengths follows from Lemma \ref{th:dom1}. Finally,
the second equality in (\ref{EQ_Z}) is obtained by taking the limit in the
last equality in (\ref{EQ_DIS_R}), and the equality between queue
lengths follows again from Lemma \ref{th:dom1}.
\end{proof}

\bigskip

Next we analyze properties of the coalescence time. Define%
\[
T_{-}=\sup\{T_{k}^{0}\leq0:\inf_{T_{k}^{0}\leq t\leq0}[Z_{T_{k}^{0}}%
(t-T_{k}^{0};Q_{v}\left(  T_{k}^{0}\right)  ,\mathcal{S}\left(  U\left(
T_{k}^{0}\right)  \right)  ,0)-Z_{T_{k}^{0}}(t-T_{k}^{0};0,0,0)]=0\}.
\]
By time reversibility we have that $\left\vert T_{-}\right\vert $ is equal in
distribution to
\[
T=\inf\{T_{k}^{0}\geq0:\inf_{0\leq t\leq T_{k}^{0}}[Z_{0}(t;Q_{v}\left(
0\right)  ,\mathcal{S}\left(  U\left(  0\right)  \right)  ,0)-Z_{0}%
(t;0,0,0)]=0\}.
\]
We next establish that $E[T]<\infty$. This result, except for the
verification of Fact I, which will be discussed in Section \ref{sec:algorithm}%
, completes the proof of Theorem \ref{Thm_Main_1}.

\begin{proposition}
\label{prop:termination} If $E[V_{n}]<cE[A_{n}]$ for $n\geq1$ and Assumption
(A1) holds,
\[
E[T]<\infty.
\]

\end{proposition}

\begin{proof}[Proof of Proposition \ref{prop:termination}]
Define%
\[
\tau=\inf\left\{  n\geq1:W_{1}\left(  T_{n}^{0};W_{v}\left(  1\right)
\right)  =W_{1}\left(  T_{n}^{0};0\right)  \right\}  .
\]
By Wald's identity, $EA_{n}<\infty$, for any $n\geq1$, it suffices to show
that $E[\tau]<\infty$. We prove the proposition by considering a
sequence of events which happens with positive probability and leads to the
occurrence of $\tau$. The idea follows from the proof of Lemma 2.4 in Chapter
XII of \cite{Asm:2003}. As $E[V_{n}]<cE[A_{n}]$, for $n\geq2$,we can find
$m,\epsilon>0$ such that for every $n\geq2$, the event $H_{n}=\{V_{n}%
<cm-\epsilon,A_{n}>m\}$ is nontrivial in the sense that $P\left(
H_{n}\right)>\delta$ for some $\delta>0$. Moreover, because $V_{1}$ and $A_{1}$ are
independent with continuous distribution we also have that $P\left(
H_{1}\right)>\delta$. Now, pick $K>cm$, large enough, and define
\[
\Omega=\left\{  W_{1}^{\left(  c\right)  }\left(  T_{k}^{0};W_{v}\left(
1\right)  \right)  \leq K\right\}  \bigcap_{n=k}^{k+\lceil cK/\epsilon
\rceil+c}H_{n}.
\]
In what follows, we first show that if $\Omega$ happens, the two bounding
systems would have coalesced by the time of the $(k+c+\lceil cK/\epsilon
\rceil)$-th arrival. We then provide an upper bound for $E[\tau]$. Let
\[
\tilde{W}_{k}=W_{1}\left(  T_{k}^{0};W_{v}\left(  1\right)  \right).
\]
For $n\geq k$, define $\tilde{V}_{n}=cm-\epsilon$, $\tilde{A}_{n}=m$.
and the (auxiliary) Kiefer-Wolfowitz sequence
\[
\tilde{W}_{n+1}=\mathcal{S}\left(  \left(  \tilde{W}_{n}+\tilde{V}%
_{n}\mathbf{e}_{1}-\tilde{A}_{n}\mathbf{1}\right)  ^{+}\right).
\]
Then $\Omega$ implies $V_{n}\leq\tilde{V}_{n}$ and
$A_{n}>\tilde{A}_{n}$, for $n\geq k$, which in turn implies $W_{1}\left(
T_n^0;W_{v}\left(  1\right)  \right)  \leq\tilde{W}_{n}$. Moreover, It is easy to
check that $\tilde{W}_{n}^{\left(  1\right)  }=0$ and $\tilde{W}_{n}^{\left(
c\right)  }<cm$ for $n=k+\lceil cK/\epsilon\rceil+1,\cdots,k+\lceil
cK/\epsilon\rceil+c$. Then, $W_{1}^{\left(  1\right)  }\left(T_n^0;W_{v}\left(
1\right)  \right)  =0$ and $W_{1}^{\left(  c\right)  }\left(T_n^0;W_{v}\left(
1\right)  \right)  <cm$ for $n=k+\lceil cK/\epsilon\rceil+1,\cdots,k+\lceil
cK/\epsilon\rceil+c$. This indicates that under $\Omega$, all the arrivals
between the $(k+\lceil cK/\epsilon\rceil+1)$-th arrival and the $(k+\lceil
cK/\epsilon\rceil+c)$-th arrival (included) have zero waiting time (enter
service immediately upon arrival), and the customers initially seen by the
$(k+\lceil cK/\epsilon\rceil+1)$-th arrival would have had left the system by
the time of the $(k+\lceil cK/\epsilon\rceil+c)$-th arrival. The same analysis
clearly holds assuming that we replace $W_{1}\left(T_k^0;W_{v}\left(  1\right)
\right)  $ by $W_{1}\left(T_k^0;0\right)  $ throughout the previous discussion.
Therefore, by the time of the $(k+\lceil cK/\epsilon\rceil+c)$-th arrival, the
two bounding systems would have exactly the same set of customers with exactly
the same remaining service times, which is equal to their service times minus
the time elapsed since their arrival times (since all of them start service
immediately upon arrival). We also notice that since there is no customer waiting, the
sorted remaining service time at $T_{k+\lceil cK/\epsilon\rceil+c}^0$
coincide with the Kiefer-Wolfowitz vector $\tilde W_{k+\lceil cK/\epsilon\rceil+c}$.
Under Assumption (A1) and $\lambda<c\mu$, $\{W\left(T_n^0; w\left(
1\right)  \right)  :n\geq1\}$  for any fixed initial condition $w$,
is a positive recurrent Harris chain, \cite{Asm:2003}
-- in fact, positive recurrence follows even under the assumption of finite means.
Now, again under Assumption (A1), we have that
\[
E\left[  \sum_{i=1}^{c}W_{v}^{\left(  i\right)  }\left(  1\right)  \right]
<\infty,
\]
Therefore, it is straightforward to show, using a standard linear
Lyapunov function, that for $K$ large enough, if $\tilde{\kappa}_{1}%
:=\inf\{n\geq1:W^{\left(  c\right)  }\left(T_n^0;W_{v}\left(  1\right)  \right)
\leq K\}$, then $E[\tilde{\kappa}_{1}]<\infty$. Moreover, for $i\geq2$, define
\[
\tilde{\kappa}_{i}:=\{n>\tilde{\kappa}_{i-1}+\lceil cK/\epsilon\rceil
+c:W^{\left(  c\right)  }\left(  n;W_{v}\left(  1\right)  \right)  \leq K\}.
\]
By the positive recurrence of the Kiefer-Wolfowitz vector, we can find a
constant $M>0$ such that
$$E[\tilde{\kappa}_{i}-\tilde{\kappa}_{i-1}]<M.$$
In this way, we can split the process into cycles (not necessarily regenerative)
$[\tilde{\kappa}_{i},\tilde{\kappa}_{i+1})$ for $i\geq1$, and the initial
period $[1,\tilde{\kappa}_{1})$. We denote $\Omega_{i}=\bigcap_{n=\tilde
{\kappa}_{i}}^{\tilde{\kappa}_{i}+\lceil cK/\epsilon\rceil+c}H_{n}$ for
$i=1,2,\cdots$. Since $P(H_{n})>\delta$, $P(\Omega_{i}
)\geq\delta^{\lceil cK/\epsilon\rceil+c+1}>0$. Let $N=\inf\{i\geq1:I\left(
\Omega_{i}=1\right)  \}$ (i.e. the first $i$ for which $\Omega_{i}$ occurs),
then $E[N]\leq\delta^{-(\lceil cK/\epsilon\rceil+c+1)}<\infty$. By Wald's
identity we have (setting $\tilde{\kappa}_{0}=0$) that
\begin{align*}
E[\tau]  &  \leq E[\tilde{\kappa}_{N}]+\lceil cK/\epsilon\rceil+c\\
&  =E\sum_{i=1}^{N}(\tilde{\kappa}_{i}-\tilde{\kappa}_{i-1})+\lceil
cK/\epsilon\rceil+c\\
&  \leq E[N]\times M+E[\tilde\kappa_{1}]+\lceil cK/\epsilon\rceil+c.
\end{align*}
\end{proof}

\section{Fact I: Simulation of Stationary Vacation System Backwards in
Time\label{sec:algorithm}}

In this section, we address the validity of Fact I, namely, that we can
simulate the vacation system backwards in time, jointly with $\left\{
T_{n}^{i}:m\leq n\leq-1\right\}  $ for $1\leq i\leq c$ for any $m\leq-1$.

Let $G_{e}(\cdot)=\lambda\int_{0}^{\cdot}\bar{G}(x)dx$ and $F_{e}(\cdot
)=\mu\int_{0}^{\cdot}\bar{F}(x)dx$ denote equilibrium CDF's of the
interarrival time and service time distributions respectively. We first notice
that simulating the stationary arrival process $\left\{  T_{n}^{0}%
:n\leq-1\right\}  $ and stationary service/vacation completion process
$\left\{  T_{n}^{i}:n\leq-1\right\}  $ for each $1\leq i\leq c$ is
straightforward by the reversibility of $\mathcal{T}_{n}^{i}$ for $0\leq i\leq
c$. Specifically, we can simulate the renewal arrival process forward in time
from time $0$ with the first interarrival time following $G_{e}$ and
subsequence interarrival times following $G$. We then set $T_{-k}^{0}%
=-T_{k}^{0}$ for all $k\geq1$. Likewise, we can also simulate the
service/vacation completion process of server $i$, for $i=1,\dots,c$, forward
in time from time $0$ with the first service/vacation completion time
following $F_{e}$ and subsequent service/vacation requirements distributed as
$F$. Let $T_{k}^{i}$ denote the $k$-th service/vacation completion time of
server $i$ counting forwards (backwards) in time. Then we set $T_{-k}%
^{i}=-T_{k}^{i}$.

Similarly, we have the equality in distribution, for all $t\geq0$ (jointly)%
\[
X_{-t}\left(  0\right)  =X_{0}\left(  t\right)  ,
\]
therefore we have from (\ref{Stat_Backwards}) that the following equality in
distribution holds for all $t\geq0$ (jointly)%
\[
Q_{v}(-t)=\sup_{s\geq t}X_{0}\left(  s\right)  -X_{0}\left(  t\right)  .
\]
The challenge in simulating $Q_{v}(-t)$, involves in sampling $M(t)=\max
_{s\geq t}\{X_{0}(s)\}$ jointly with $N_{0}^{i}\left(  t\right)  $ for $1\leq
i\leq c$ during any time interval of the form $[0,T]$ for $T>0$. If we can do
that, then we can evaluate
\[
X_{0}(t)=N_{0}^{0}(t)-\sum_{i=1}^{c}N_{0}^{i}(t),
\]
and, therefore, $Q_{v}(-t)=M(t)-X(t)$.

In what follows we first introduce a trick to decompose the process of
sampling $M\left(  t\right)  $ into that of sampling the maximum of $c+1$
independent negative drifted random walks.

Choose $a\in(\lambda,c\mu)$. Then
\[
X(t)=(N_{0}^{0}(t)-at)+\sum_{i=1}^{c}\left(  \frac{a}{c}t-N_{0}^{i}(t)\right)
.
\]
We define $(c+1)$ negative drifted random walks as follows:
\[
S_{0}^{(0)}=0,~~~~S_{n}^{(0)}=S_{n-1}^{(0)}+(-aA_{n}+1)\mbox{ for
$n=1,2,\dots$},
\]
and for $i=1,\dots,c$,
\[
S_{0}^{(i)}=\frac{a}{c}V_{1}^{(i)},~~~S_{n}^{(i)}=S_{n-1}^{(i)}+\left(
-1+\frac{a}{c}V_{n+1}^{(i)}\right)  \mbox{ for $n=1,2,\dots$}.
\]

Figure \ref{fig:RW} plots the relationship between $\{N_{0}^{0}(t)-at:t\geq
0\}$ and $\{S_{n}^{(0)}:n\geq0\}$, and the relationship between $\{\frac{a}%
{c}t-N_{0}^{i}(t):t\geq0\}$ and $\{S_{n}^{(i)}:n\geq0\}$ for $i=1,\dots,c$.
From Figure \ref{fig:RW}, it is easy to check that
\[
\max_{s\geq t}\{N_{0}^{0}(s)-as\}=\max\{N_{0}^{0}(t)-at,\max_{n\geq N_{0}%
^{0}(t)+1}\{S_{n}^{(0)}\}\},
\]
and for $i=1,\dots,c$,
\[
\max_{s\geq t}\left\{  \frac{a}{c}s-N_{0}^{i}(s)\right\}  =\max\{\frac{a}%
{c}t-N_{0}^{i}(t),\max_{n\geq N_{0}^{i}(t)}\{S_{n}^{(i)}\}\}.
\]
\begin{figure}[tbh]
\caption{The relationship between the renewal processes and the random walks}%
\label{fig:RW}%
\centering
\subfloat[$N_0^0(t)-at$ and $S_n^{(0)}$] {\includegraphics[width=0.4\textwidth]{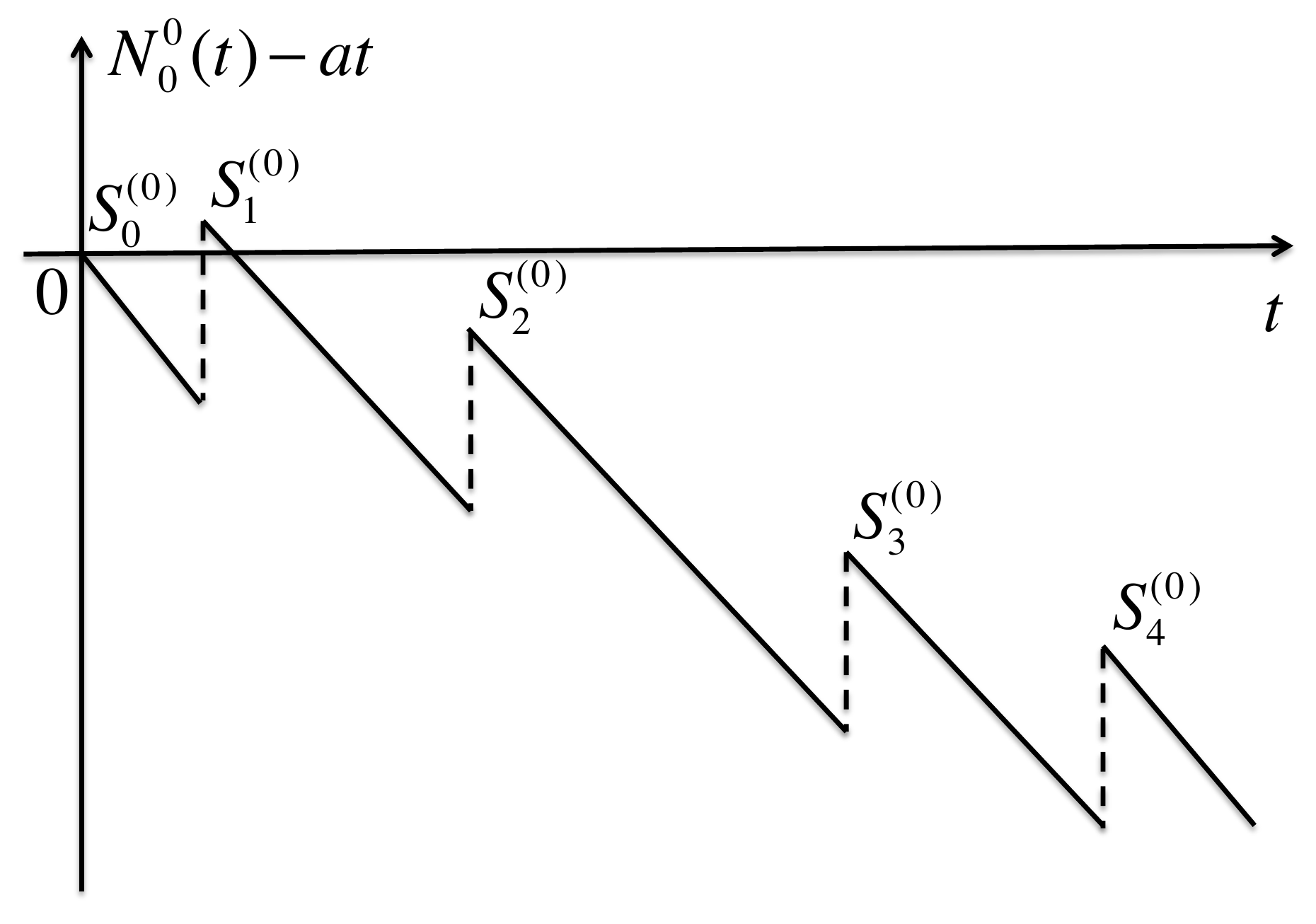} }
~~~~~ \subfloat[$(a/c) t -N_0^i(t)$ and $S_n^{(i)}$] {\includegraphics[width=0.4\textwidth]{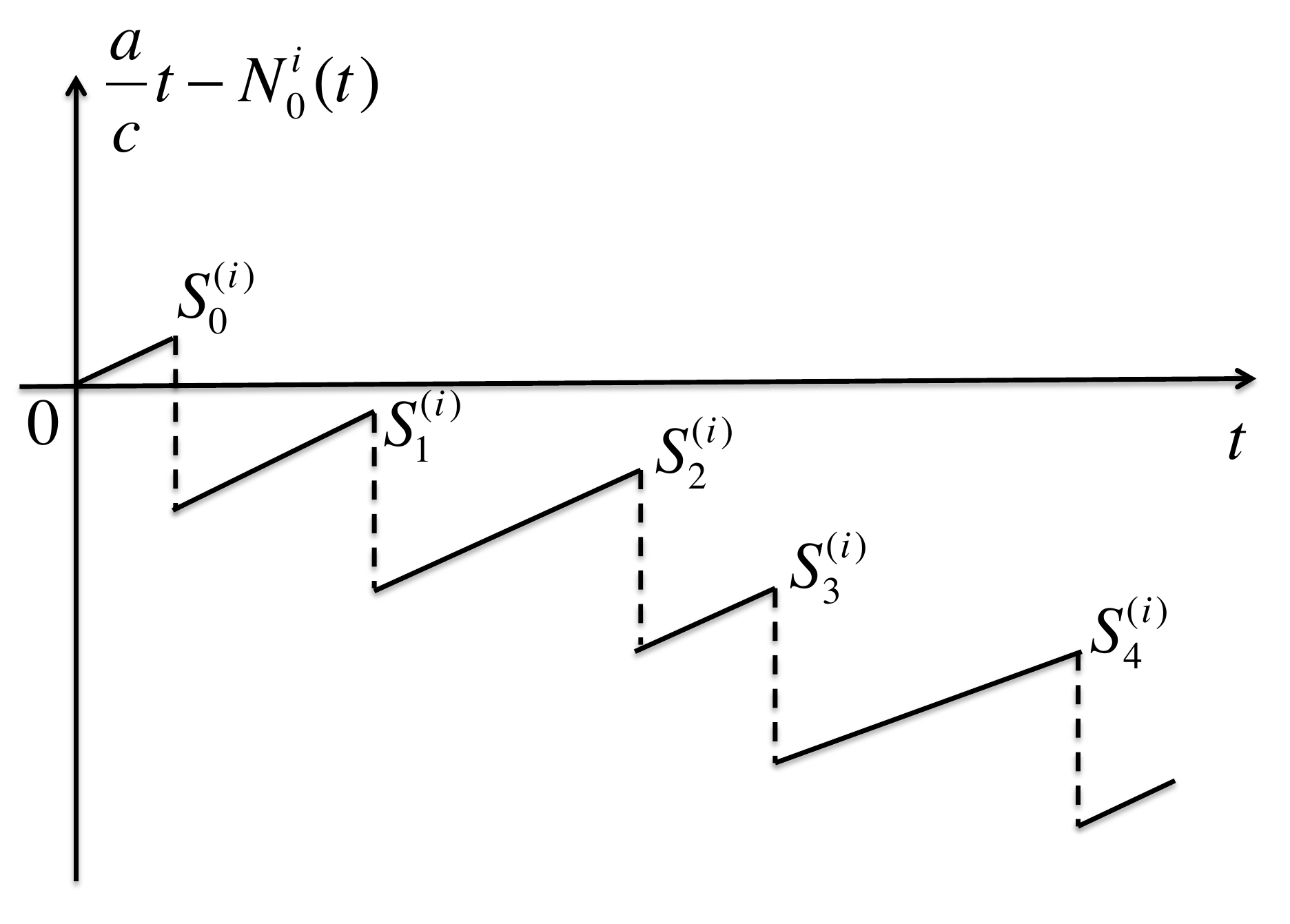}}\end{figure}

Consequently, the algorithm to simulate $M(t)=\max_{s\geq t}\{X_{0}(s)\}$
jointly with $N_{0}^{i}\left(  t\right)  $ for $1\leq i\leq c$ during any time
interval of the form $[0,T]$ for $T>0$ can be implemented if we can simulate
$S_{n}^{(i)}$ jointly with $M_{n}^{\left(  i\right)  }=\max\{S_{k}^{\left(
i\right)  }:k\geq n\}$ for $n\geq1$\ for $i=0,1,\dots,c$ .  The algorithm to generate
$S_{n}^{(i)}$ jointly with $M_{n}^{\left(  i\right)  }$ under Assumption (A1)
has been developed in \cite{BlanWall:2014}. We provide a
detailed Matlab implementation of each of the algorithms required to execute
Facts I-IV in the online appendix to this paper.

\section{Numerical experiments\label{sec:numerical}}

As a sanity check, we have implemented our
Matlab code in the case of an $M/M/c$ queue, for which the steady-state analysis
can be performed in closed form.



As a first step, we have compared the theoretical distribution to the empirical distribution
obtained from a large number of runs of our perfect simulation algorithm for
different sets of parameter values, and they are all in close agreement. As an example,
Figure \ref{fig:testNumCus1} shows the result of such test when $\lambda=3$,
$\mu=2$, $c=2$. Grey bars show the empirical result of $5,000$ draws using our
perfect simulation algorithm, and black bars show the theoretical distribution
of number of customers in system. Figure \ref{fig:testNumCus2} provides another comparison
with a different set of parameters.

\begin{figure}[ptb]
\centering
{\includegraphics[width=1\textwidth]{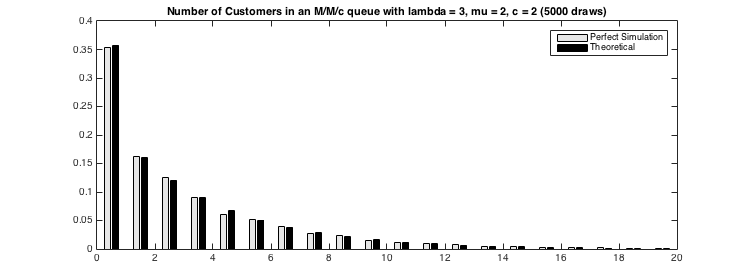} }%
\caption{Number of customers for an $M/M/c$ queue in stationarity when
$\lambda=3$, $\mu=2$ and $c=2$.}%
\label{fig:testNumCus1}%
\end{figure}

\begin{figure}[ptb]
\centering
{\includegraphics[width=1\textwidth]{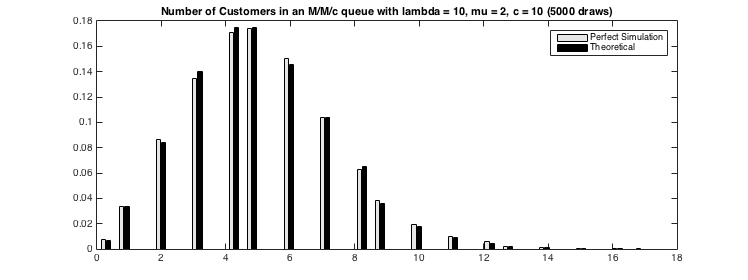}
}\caption{Number of customers for an $M/M/c$ queue in stationarity when
$\lambda=10$, $\mu=2$ and $c=10$.}%
\label{fig:testNumCus2}%
\end{figure}

Next we run numerical experiments to see how the running time of our
algorithm, measured by mean coalescence time of two bounding systems, scales
as the number of servers grows and the traffic intensity $\rho$ changes.
Starting from time $0$, the upper bound queue has its queue length sampled
from the theoretical distribution of an $M/M/c$ vacation system and all
servers busy with remaining service times drew from the equilibrium
distribution of the service/vacation time; and the lower bound queue is empty. Then we
run both the upper bound and lower bound queues forward in time with the same
stream of arrival times and service requirements until they coalescence.
Table \ref{table:1} shows the estimated average coalescence time,
$E[T]$, based on $5000$ iid samples, for different system scales
in the Quality driven regime (QD). We observe that $E[T]$ does not
increase much as the system scale parameter, $s$, grows. Table \ref{table:2} shows similar
results for the Quality-and-Efficiency driven operating regime (QED). In this
case, $E[T]$ increases at a faster rate with $s$ than the QD case, but the
magnitude of increment is still not significant.

\begin{table}[th]
\caption{simulation result for coalescence time of $M/M/c$ queue}%
\label{table:1}
\centering
{\small (QD: $\lambda_{s}=s, c_{s}=1.2s, \mu=1$)} \bigskip
\par
\renewcommand{\arraystretch}{1.5}
\begin{tabular}
[c]{||ccc||}\hline
s & mean & 95\% confidence interval\\[0.5ex]\hline\hline
100 & 6.4212 & [6.2902, 6.5522]\\\hline
500 & 7.0641 & [6.9848, 7.1434]\\\hline
1000 & 7.7465 & [7.6667, 7.8263]\\\hline
\end{tabular}
\end{table}

\begin{table}[h]
\caption{simulation result for coalescence time of $M/M/c$ queue}%
\label{table:2}
\centering
{\small (QED: $\lambda_{s}=s, c_{s}=s+2\sqrt{s}, \mu=1$)} \bigskip
\par
\renewcommand{\arraystretch}{1.5}
\begin{tabular}
[c]{||ccc||}\hline
s & mean & 95\% confidence interval\\[0.5ex]\hline\hline
100 & 6.5074 & [6.3771, 6.6377]\\\hline
500 & 8.5896 & [8.4361, 8.7431]\\\hline
1000 & 9.4723 & [9.3041, 9.6405]\\\hline
\end{tabular}
\end{table}

Finally we run a numerical experiment aiming to test how computational complexity of our algorithm changes with traffic intensity, $\rho=\lambda/c\mu$. Here we define the computational complexity as the total number of renewals (including arrivals and services/vacations) the algorithm samples in total to find the coalescence time. We expect the complexity to scale like $(c+1)(1-\rho)^{-2}E[T(\rho)]$ where $(c+1)$ is the number of renewal processes we need to simulate, $(1-\rho)^{-2}$ is on average the amount of renewals we need to sample to find its running time maximum for each renewal process, and $E[T(\rho)]$ is the mean coalescence time when the traffic intensity is $\rho$. Table \ref{table:3} summarizes our numeral results, based 5000 independent runs of the algorithm for each $\rho$. We run the coalescence check at $10\times 2^k$ for $k=1,2,\dots$, until we find the coalescence. We observe that as $\rho$ increase, the computational complexity increases significantly, but when multiplied by $(1-\rho)^2$, the resulting products are of about the same magnitude - up to a factor proportional to $\lambda$, given that the number of arrivals scales as $lambda$ per unit time. Therefore, the main scaling parameter for the complexity here is $(1-\rho)^{-2}$. Notice that if we simulate the system forward in time from empty, it also took around $O\left((1-\rho)^{-2}\right)$ arrivals to get close to stationary.

\begin{table}[th]
\caption{simulation result for computational complexities with varying traffic intensities}
\label{table:3}
\centering
{\small $M/M/c$ queue with fixed $\mu=5$ and $c=2$} \bigskip
\par
\renewcommand{\arraystretch}{1.5}
\begin{tabular}{ | P{.5cm} | P{9em}| P{8em} | P{11em} | P{13em} | } \hline
$\lambda$ & traffic intensity $\left(\rho\right)$ & mean number of renewals sampled & mean index of successful inspection time & mean number of renewals sampled $\times\left(1-\rho\right)^2$ \\
\hline
5 & 0.5 & 225.6670 & 11.7780 & 56.4168 \\
\hline
6 & 0.6 & 377.0050 & 14.7780 & 60.3208 \\
\hline
7 & 0.7 & 764.3714 & 21.9800 & 68.7934\\
\hline
8 & 0.8 & 2181.3452 & 44.2320 & 87.2538\\
\hline
9 & 0.9 & 12162.6158 & 161.0840 & 121.6262\\
\hline
\end{tabular}
\end{table}

\bibliographystyle{plain}
\bibliography{exact_ref}

\end{document}